\documentclass[reqno,11pt]{amsart}
\usepackage{amsmath, latexsym, amsfonts, amssymb, amsthm, amscd, stmaryrd}
\usepackage{graphics,epsf,psfrag, color}
\numberwithin{equation}{section}
\allowdisplaybreaks
\usepackage{bbm, dsfont}
\usepackage{mathtools}
\usepackage{xcolor}

\newtheorem{theorem}{Theorem}[section]
\newtheorem{lemma}[theorem]{Lemma}
\newtheorem{proposition}[theorem]{Proposition}

\newtheorem{rem}[theorem]{Remark}

\newcommand{\RN}[1]{%
  \textup{\uppercase\expandafter{\romannumeral#1}}%
}

\makeatletter
\def\captionfont@{\footnotesize}
\def\captionheadfont@{\scshape}

\long\def\@makecaption#1#2{%
  \vspace{2mm}
  \setbox\@tempboxa\vbox{\color@setgroup
    \advance\hsize-6pc\noindent
    \captionfont@\captionheadfont@#1\@xp\@ifnotempty\@xp
        {\@cdr#2\@nil}{.\captionfont@\upshape\enspace#2}%
    \unskip\kern-6pc\par
    \global\setbox\@ne\lastbox\color@endgroup}%
  \ifhbox\@ne 
    \setbox\@ne\hbox{\unhbox\@ne\unskip\unskip\unpenalty\unkern}%
  \fi
  \ifdim\wd\@tempboxa=\z@ 
    \setbox\@ne\hbox to\columnwidth{\hss\kern-6pc\box\@ne\hss}%
  \else 
    \setbox\@ne\vbox{\unvbox\@tempboxa\parskip\z@skip
        \noindent\unhbox\@ne\advance\hsize-6pc\par}%
\fi
  \ifnum\@tempcnta<64 
    \addvspace\abovecaptionskip
    \moveright 3pc\box\@ne
  \else 
    \moveright 3pc\box\@ne
    \nobreak
    \vskip\belowcaptionskip
  \fi
\relax
}
\makeatother
\def\writefig#1 #2 #3 {\rlap{\kern #1 truecm
\raise #2 truecm \hbox{#3}}}

\renewcommand{\tilde}{\widetilde}
\renewcommand{\hat}{\widehat}

\newcommand{\cF}{\ensuremath{\mathcal F}}

\newcommand{\cP}{\ensuremath{\mathcal P}}

\newcommand{\llb}{\llbracket} 
\newcommand{\rrb}{\rrbracket}

\newcommand{\N}{{\ensuremath{\mathbb N}} } 
 
\renewcommand{\P}{{\ensuremath{\mathbb P}} }

\newcommand{\eps}{\varepsilon}

\renewcommand{\O}{\Omega}

\newcommand{\1}[1]{\mathbf{1}_{\left\lbrace   #1   \right\rbrace }}

\renewcommand{\t}{\tau}

\renewcommand{\a}{\alpha}

\renewcommand{\a}{\alpha}
\newcommand{\gb}{\beta}
\newcommand{\gga}{\gamma}            
\renewcommand{\d}{\delta}

\newcommand{\w}{\omega}

\newcommand{\bP}{{\ensuremath{\mathbf P}} }
\newcommand{\bE}{{\ensuremath{\mathbf E}} }

\newcommand{\E}{\mathbb{E}}
\newcommand{\Z}{\mathbb{Z}}
\newcommand{\R}{\mathbb{R}}

\newcommand{\g}{\gamma}
\renewcommand{\b}{\beta}

\begin{document}
\title{Directed polymer in $\g$-stable Random Environments}


\author[Roberto Viveros]{Roberto Viveros}
 \address{Roberto Viveros \hfill\break
IMPA\\
Estrada Dona Castorina, 110\\ Rio de Janeiro 22460-320 \\ Brazil.}
\email{rviveros@impa.br}


\begin{abstract}
{The transition from a weak-disorder (diffusive phase) to a strong-disorder (localized phase) for directed polymers in a random environment is a well studied phenomenon. In the most common setup, it is established that the phase transition is trivial when the transversal dimension $d$ equals $1$ or $2$ (the diffusive phase is reduced to $\beta=0$) while when $d\ge 3$, there is a critical temperature $\beta_c\in (0,\infty)$ which delimits the two phases.
The proof of the existence of a diffusive regime for $d\ge 3$ is based on a second moment method \cite{bolthausen, comets, Imbrie}, and thus relies heavily on the assumption that the variable which encodes the disorder intensity 
(which in most of the mathematics literature assumes the form $e^{\beta \eta_x}$), has finite second moment. The aim of this work is to investigate how the presence/absence of phase transition may depend on the dimension $d$ in the case when the disorder variable displays heavier tail. To this end we replace $e^{\beta \eta_x}$ by $(1+\beta \omega_x)$ where $\omega_x$ is in the domain of attraction of a stable law with parameter $\gga \in (1, 2)$.}
 In this setup we show that we have a non-trivial phase transition if and only if 
 $\g > 1 + 2/d$. More precisely, when $\g \leq 1 + 2/d$, the free-energy of the system { is smaller than its annealed counterpart } at every temperature whereas when $\g > 1 + 2/d$ { the  martingale sequence of renormalized  partition functions} converge to an almost surely  positive random variable for $\beta$ sufficiently small.
\end{abstract}

\maketitle

	\section{Introduction}
\subsection{The model}	
A directed polymer system consists in a random distribution of walks or paths in $\Z^d$ parametrized by time. The graph of the walk in $\Z^{d+1}$ is the \textit{polymer} which stretches in the time direction and so is called \textit{directed}. We consider walks interacting with a random space-time environment, with power-law distribution and we show bounds on the free energy at high temperature. Directed polymers in a random environment have appeared originally in the physics literature as an effective model for the interface in two-dimensional Ising model with random exchange interactions \cite{huse} and has become an interesting subject of study for many authors ever since (see \cite{review, saintf} for a review on the matter).

Consider the following version of the directed polymer model {(we opted to introduce the model first in the more conventional setup with the environment randomness appearing in exponential form as it is the more convenient option when referring to the existing literature. We introduce and justify our modified setup in Section \ref{our work})}:let $\bP_x$ be the probability measure on the space $\left( \O, \mathcal{F} \right)  := \left( {\left( \Z^d\right) }^{\N}, \cP(\Z^d)^{\otimes \N}\right) $ of sequences $S:=(S_n)_{n\geq 0}$ such that:
\begin{equation}
\begin{split}\label{first}
&S_0 = x,\\
&\{S_{n} - S_{n-1}\}_{n\geq 1}\text{ is an IID sequence, and } \\
&\bP_x [S_1 = x + e_j] = \bP_x [S_1 = x - e_j] = \frac{1}{2d},
\end{split}
\end{equation}
for all $j\leq d$ where $\{e_1,...,e_d\}$ is the canonical basis of $\R^d$. The set of points $\{(n,S_n): n \geq 0 \} \subset \Z^{1+d}$ represents the graph of a simple random walk on $\mathbb{Z}^d$.

Independently, also consider a sequence of IID random variables $\eta := \{\eta_{n,z}: n \in \N, z \in \Z^{d}\}$, called \textit{ the environment}, defined on a probability space $(\Lambda, \cF, \P)$, that satisfies,
\begin{equation} \label{condition}
\begin{split} 
&\E[\eta_{0,0}] =0 \text{ and} \\
& \E[\exp (\b \eta_{0,0})] <  \infty, \text{ for all } \b \in \R.  
\end{split}
\end{equation}
For a given $\beta > 0$, $N  \in \N$ and a fixed realization of the environment $\eta$, we define the measure $\bP_{N}^{\beta,\eta}$ on the space $\O$, called the \textit{polymer measure}, by its Radon-Nikodym derivative with respect to $\bP_0$:
\begin{equation} \label{poly}
\dfrac{\mathrm{d} \bP_{N}^{\beta,\eta}}{\mathrm{d} \bP_0}(S) = \dfrac{1}{Z_{N}^{\beta,\eta}} \exp \left(\beta \sum_{n=1}^{N} \eta_{n,S_n} \right),  
\end{equation}
where $Z_{N}^{\beta,\eta}$ is the positive normalization factor that makes $\bP_{N}^{\beta,\eta}$ a probability measure. 
We call $Z_{N}^{\beta,\omega}$ the \textit{partition function} of the system and its value is given by
\begin{equation}
Z_{N}^{\beta,\eta} = \bE_0 \left[ \exp \left(\beta \sum_{n=1}^{N} \eta_{n,S_n} \right) \right] = (2d)^{-N} \sum_{S \in \O_N}  \exp \left(\beta \sum_{n=1}^{N} \eta_{n,S_n} \right),
\end{equation}
where 
\begin{equation}
\O_N := \left\lbrace S \in \Z^N: S_0 = 0, |S_n - S_{n-1}|=1, \forall n \in [1,N] \cap \Z \right\rbrace.
\end{equation}
The goal for this model is to study how the presence of the environment affects the distribution of the random walk. Intuitively, this new measure $\bP_{N}^{\beta,\eta}$ rewards (penalizes) walks that visits sites with higher (smaller) values of the environment. The parameter $\b$ (the inverse temperature) is used to increase or decrease the possible influence of the environment over the measure $\bP_{N}^{\beta,\eta}$. Notice that when $\gb=0$, $\bP_{N}^{\beta,\eta}$ becomes $\bP_0$.


\subsection{Known facts}
In \cite{bolthausen}, Bolthausen observed that the renormalized partition function 
\begin{equation}
W_N^{\beta,\eta} := \frac{Z_{N}^{\beta,\eta}}{\E \left[ Z_{N}^{\beta,\eta}\right] },  
\end{equation}
is a positive martingale with respect to the sequence of $\sigma$-fields $\left\lbrace \mathcal{G}_N\ \right\rbrace _{N \geq 0}$ where $\mathcal{G}_N := \sigma\{\eta_{n,z}: 0 \leq n \leq N, z \in \Z^d \} $. By the
Martingale Convergence Theorem, it follows that the limit
\begin{equation}
W_{\infty}^{\beta,\eta} := \lim_{N \to \infty} W_{N}^{\beta,\eta},
\end{equation}
exists $\P$-\textit{a.s.} and is a non-negative random variable. The event $\{W_{\infty}^{\beta,\eta} = 0\}$ belongs to the tail $\sigma$-field of $\{\mathcal{G}_N, N \geq 0\}$. Hence, by Kolmogorov's $0-1$ Law,
\begin{align}\label{dichotomy}
&\P \left\lbrace  W_{\infty}^{\beta}  > 0  \right\rbrace \in \{0,1\}. 
\end{align}
This dichotomy allows to define a natural manner to characterize the influence of disorder. Following standard terminology we say that we have \textit{weak disorder} if $W_{\infty}^{\beta} > 0$ $\P$-\textit{a.s.} and \textit{strong disorder} if $W_{\infty}^{\beta} = 0$ $\P$-\textit{a.s.}.

Roughly speaking, weak disorder implies that the polymer paths have the same behavior as the simple random walk (delocalized phase). A series of papers \cite{Imbrie, bolthausen,albeverio,zhou,comets} lead to the following: Assuming $d\geq 3$ and weak disorder, the measures $\bP_N^{\b, \eta}$, after rescaling, converge in law to the Brownian Motion, for almost all realizations of the environment.

On the other hand, strong disorder implies that the polymer is largely influenced by the disorder and is attracted to sites with
favorable environment (localize phase). 
We mention \cite[Theorem 2.1]{cometsyoshido}, where it is been shown that for $\b > 0$, 
\begin{equation}
\left\lbrace W_{\infty}^{\beta, \eta} = 0  \right\rbrace = 
\left\lbrace \sum_{n \geq 1} \left( \bP_{n-1}^{\b, \eta}\right)^{\otimes 2}  [S_n = S_n'] = \infty \right\rbrace \text{   }\P\text{-a.s.,}
\end{equation}
where $S$ and $S'$ are two independent polymers with distribution $\bP_{n-1}^{\b, \eta}$. Moreover, if $\P[W_{\infty}^{\beta, \eta} = 0] = 1$, then there exists some constants $c_1$, $c_2 \in (0,\infty)$ such that,
\begin{equation}
-c_1 \log W_N^{\b,\eta} \leq  \sum_{n \geq 1}^N \left( \bP_{n-1}^{\b, \eta}\right)^{\otimes 2}  [S_n = S_n'] \leq -c_2 \log W_N^{\b,\eta},
\end{equation}
for $N$ large enough, $\P$-a.s. This result suggests that when we have strong disorder, the polymer is more attracted to sites with favorable environment and the probability of two of them to occupy the same last site increases (recall that for the simple random walk, $\bP_0[S_n = S_n']\sim \frac{C_d}{n^{d/2}}$). Also the decay property of $W_N$ is reflected in some specific localization property of the path.

In \cite{comets}, it was also shown that there exists a critical value $\b_c  = \b_c(d) \in [0,\infty]$ with
\begin{align}
\b_c &= 0 \text{   for }d = 1,2 \text{ and} \\
\b_c &> 0 \text{   for }d \geq 3,
\end{align}
such that there is weak disorder for $\b \in [0,\b_c)$ and strong disorder for $\b > \b_c$.

\subsection{Free energy.} 

A lot of information about the model is encoded in the following quantity
\begin{equation} \label{freenergy2}
p(\beta) := \lim_{N \to \infty} \dfrac{1}{N}   \log  W_{N}^{\beta,\eta} ,
\end{equation}
called the \textit{free energy} of the model. This limit exists and is non-random \cite[Proposition 2.5]{cometsyoshido}. Moreover, the function $\b \mapsto p(\b)$ is continuous and non-increasing. In particular, there exists
$\bar{\beta}_c = \bar{\beta}_c(d)$ with
\begin{equation}
0 \leq \beta_c \leq \bar{\beta_c} \leq \infty, 
\end{equation}
such that
\begin{equation}
p(\beta) = 
\begin{cases}
= 0       & \quad \text{if } \beta \leq  \bar{\beta_c} \\
< 0       & \quad \text{if } \beta >  \bar{\beta_c} \\
\end{cases}
\end{equation}
Notice that if $W_\infty^{\gb,\eta} > 0$ then $p(\gb)=0$. In view of this, we say that \textit{very strong disorder} holds when $p(\gb)<0$. 

Some estimates have been proved for the free energy. In dimension $d=1$, it is known that $p(\b)$ is of order $-\b^4$ as $\b \to 0$ \cite{Lacoin, Watbled, Alexander}. In \cite{Nashimi} is been shown that, under some conditions on the environment,
\begin{equation}
\lim_{\b \to 0} \frac{p(\b)}{\b^4} = - \frac{1}{6}. 
\end{equation}
In dimension $d = 2$, it has been proved \cite{d2} that,
\begin{equation}
\lim_{\b \to 0} \b^2 \log |p(\b)| = -\pi.
\end{equation}
In particular, we have that $\b_c = \bar{\b_c} = 0,$ for $d = 1,2$.

\subsection{Our work} \label{our work}

The techniques used to prove weak disorder in dimension $d \geq 3$ relies, in a crucial way, on the boundedness of the second moment of the partition function \cite{comets}. In the present paper, we study the model in the case where the environment is IID but with a distribution belonging to the domain of attraction of a stable law with parameter $ \g \in (1, 2)$; In this case the partition function has an infinite second moment. Specifically we consider the sequence of IID random variables $\omega = \{\omega_{n,z}: n \in \N, z \in \Z^{d}\}$, that satisfies,
\begin{equation} \label{condition2}
\begin{split}
&\omega_{0,0} \geq -1 \quad \P \text{-a.s.,} \\ 
&\E[\omega_{0,0}] =0 \text{ and} \\
& \P[\w_{0,0} > x] \stackrel{x \to \infty}{\sim} C_\P x^{-\g}, \text{ for } \g \in (1,2),
\end{split}
\end{equation}
as the environment. For $\beta \in [0,1)$, $N  \in \N$ and a fixed realization of the environment $\omega$, we redefine the polymer measure $\bP_{N}^{\beta,\omega}$ as
\begin{equation} \label{poly2}
\dfrac{\mathrm{d} \bP_{N}^{\beta,\omega}}{\mathrm{d} \bP_0}(S) = \dfrac{1}{Z_{N}^{\beta,\omega}} \left(\prod_{n=1}^{N} (1 + \beta \omega_{n,S_n}) \right),
\end{equation}
where as before, the partition function $Z_{N}^{\beta,\omega}:= \bE_0 \left[\prod_{n=1}^{N} (1 + \beta \omega_{n,S_n}) \right] $. Notice that this measure is well defined since the environment is bounded below by our assumption. 
Our purpose for this model is to understand how the parameters $\b$ and $\g$, affects the measure $\bP_0$ and the existence of the localize phase. We choose to work with the expression \eqref{poly2} for our polymer measure, because we do not want changes in disorder intensity $\b$ to affect { the power-tail { exponent} of the environment's distribution $\g$ which is the parameter whose influence on the phase transition we wish to study.}

{
Our assumption $\g \in (1,2)$ makes  the second moment of the partition function infinite.  Since the second moment method plays such a crucial role in the analysis in \cite{bolthausen, Imbrie}, it is reasonable to expect that the picture differs in this case, possible due to the influence of extreme values 
of field $\omega$ as often observed in heavy tailed setups.}
We prove that for some values of $\g$, we have strong disorder for all $\b >0$, in all dimensions: specifically, for $d \geq 3$, there is a critical value $\g_c = \g_c(d) := 1 + \dfrac{2}{d}$, such that $\g \in (1,\g_c]$ implies strong disorder, for all $ \b \in (0,1)$ and $\g \in (\g_c, 2]$ implies weak disorder, for all $\b > 0$ sufficiently small. We summarize our results bellow. The free energy can now be written as 
\begin{equation} \label{fe}
p(\b) := \lim_{N \to \infty} \frac{1}{N} Z_{N}^{\beta,\omega},
\end{equation}
since $\E \left[Z_{N}^{\beta,\omega}\right] = 1$. The proof of the convergence of \eqref{fe} follows the same lines as \cite[Proposition 2.5]{cometsyoshido} for which is omitted in this manuscript. Nevertheless, we show continuity and monotonicity of $\b \to p(\b)$ in Theorem \ref{freenergy}. From now on, we assume the environment always satisfies \eqref{condition2} and unless otherwise specified, the polymer measure/partition function is the one defined in \eqref{poly2}.

\begin{theorem} \label{1st}
	When the environment satisfies the condition \eqref{condition2} and if $ \g  \leq \g_c$, very strong disorder holds, for all values of $\b >0$, in all dimensions $d \geq 1$. In particular, if $d \geq 3$ and $ \g  < \g_c$,		
	\begin{equation}\label{alpha}
	\lim_{\b \to 0} \frac{\log |p(\b)|}{\log \b} = \a,
	\end{equation}
	where $\a = \a(d,\g) := \frac{\g(\g_c -1)}{\g_c - \g}$.	Also, if $ \g  = \g_c$, we have that,
	\begin{equation}\label{inf}
	\lim_{\b \to 0} \frac{\log |p(\b)|}{\log \b} = \infty.
	\end{equation}
\end{theorem}
{ \begin{rem}
        We obtain in fact much more quantitative upper bound statements concerning $p$ beyond \eqref{alpha}-\eqref{inf} (detailed in the core part of the manuscript). In particular we prove that $p(\beta)\le -c\gb^{\alpha}$ in a neighborhood of zero, and we believe that this upper bound sharply describes (up to constant factor) the asymptotic behavior. 
       \end{rem}}

\begin{theorem} \label{2nd}
	Assuming the same conditions \eqref{condition2} for the environment and if $ \g  > \g_c$, we have that $\b_c > 0$ for dimensions $d \geq 3$.
\end{theorem} 

\subsection{Related works}
Among other works that deal with heavy tail environments we mention \cite{berger}, where in the setup \eqref{poly} the environment $\eta$ is allowed to belong to the domain of attraction of a $\a$-stable law and it is studied properties of paths trajectories drawn from the polymer measure. In this context there is no free energy so the work is fundamentally different from ours. In \cite{comets long jump}, it is studied the influence of the jump distribution on the delocalization-localization transition and the interplay between jump tails, spatial dimension and existence of the delocalized phase, when nearest neighbor walks are replaced by long range jumps. Our results most likely extend to that setup, the criterion for having a weak disorder phase in dimension $d=1$ becoming $\g > \gamma_c= 1+ \alpha$ where $\alpha \in (0,1)$ is the exponent of the random walk. 
We also mention \cite{brownian} as another case where a change in the environment setup (there is: moving from the IID setup to a strong spatial correlation in the environment) modifies the criterion for having no phase transition. 

\subsection{Organization of the paper}
We show upper bounds for the free energy in Sections 2 and 3, for the cases $ \g  < \g_c(d)$ and $ \g  = \g_c(d)$ respectively. In Section 4 we bound some fractional moments of the partition function, when $ \g  > \g_c(d)$. This leads to Theorem \ref{2nd} through a uniform integrability argument. In Section 5 we show a lower bound for the free energy when $ \g  < \g_c(d)$. This completes the proof of \ref{1st}.

\subsection{Notation}
For simplicity, we write $\bE[\cdot]$ and $\bP[\cdot]$ instead of $\bE_0[\cdot]$ and $\bP_0[\cdot]$ for the law of the simple random walk, starting form the origin. We also sometimes omit brackets from expectations when it is clear from the context with respect to which random variable it is integrating. For example, we may write $\bE[(\cdot)^q]$ as $\bE(\cdot)^q$. Also, to avoid ambiguities, $\bE[\cdot]^q$ always means $(\bE[\cdot])^q$.

\subsection{Acknowledgement}
The author is very grateful to Hubert Lacoin for suggesting to work with this problem, for very fruitful discussions and for his very useful comments and suggestions concerning the manuscript.

\section{Disorder relevance.} \label{1}

In this section we show an upper bound for the free energy that is required for the first limit of Theorem \ref{1st}. We base upon the proof of the upper bound of \cite[Theorem 1.4]{Lacoin}, where an analogous bound is proved, in the setup \eqref{poly}. The proof combines coarse graining, a fractional moment method and a different idea for the change of measure: we penalize sites which values are above a certain threshold. These ideas have appeared originally in \cite{hubertbound} for the pinning model and in \cite{toninelli} for the copolymer model.   
\begin{proposition}
	Assuming the environment's distribution satisfies condition \eqref{condition2} and $\g < \g_c(d)$ we have,
	\begin{equation}\label{1stupperbound}
	p(\b) \leq - C \b^\a,
	\end{equation}
	where $\a = \frac{\g(\g_c -1)}{\g_c - \g}$, for some positive constant $C$ and all $\b$ sufficiently small. 
\end{proposition}
\begin{proof}
	Fix $n \in \N$ and $\theta \in (0,1)$. By Jensen's Inequality,
	\begin{eqnarray}
	p(\beta)\leq \lim_{m \to \infty} \frac{1}{n m\theta} \log \E  \left( Z_{mn}^{\beta,\omega}\right)^\theta.
	\end{eqnarray}
	Notice that we replace the expectation of a logarithm by the estimation of a fractional moment, which in principle should be easier to handle.
	The goal is to prove that, for all $m \in \N$, 
	\begin{eqnarray}\label{goal}
	\E  \left( Z_{mn}^{\beta,\omega}\right) ^\theta \leq \exp(-m), 
	\end{eqnarray}
	for some convenient value of $n$. Specifically, let $n$ be a squared integer such that $\frac{C_1}{2 n}\leq \beta^\alpha < \frac{C_1}{n}$, where $C_1 > 0$ is a constant to be defined later, then \eqref{goal} implies
	\begin{eqnarray}
	p(\beta) \leq - \frac{\beta^\alpha}{C_1 \theta}.
	\end{eqnarray}
	 We first decompose the partition function $Z_{mn}^{\beta,\omega}$ according to the position of the walk at times $n,2n,3n,...,mn$:
	
	\begin{eqnarray}
	Z_{mn}^{\beta,\omega} = \sum_{y_1,...,y_m \in \Z^d} \bE \left[ \prod_{i=1}^{nm}
	(1 + \beta \omega_{i,S_i})
	\prod_{k=1}^{m} \1{S_{kn} \in I_{y_k}} \right], 
	\end{eqnarray}
	where the region 
	\begin{eqnarray}
	I_z :=\{x=(x_1,...,x_d) \in \Z^d : z_j \sqrt{n} \leq x_j < (z_j+1) \sqrt{n} \quad \forall j \leq d  \},
	\end{eqnarray} 
	is defined for any $z = (z_1,...,z_d) \in Z^d$.	Using the inequality
	\begin{equation}\label{jensen}
	\left( a_1 + \dots + a_t\right)^\theta \leq  a_1^\theta + \dots + a_t^\theta,
	\end{equation}
	which holds for any $\theta \in [0,1]$ and $a_i \geq 0$, we deduce
	\begin{eqnarray} \label{coarse}
	\E   \left( Z_{mn}^{\beta,\omega}\right) ^\theta   \leq \sum_{y_1,...,y_m \in \Z^d} \E\left(  \tilde{Z}_{y_1,...,y_m} \right)^\theta,
	\end{eqnarray}
	where
	$$\tilde{Z}_{y_1,...,y_m} := \bE \left[ \prod_{i=1}^{n m}
	(1 + \beta \omega_{i,S_i})
	\prod_{k=1}^{m} \1{S_{kn} \in I_{y_k}} \right].$$
	In order to bound the expectation of $\tilde{Z}_{y_1,...,y_m}$, we introduce a change of measure that penalizes higher values of the environment on regions that paths are likely to visit, increasing the value of the partition function. Consider the function
	\begin{equation}
	\begin{split}
	g:\R &\to \R \\
	x &\mapsto 1 - \tfrac{1}{2} \1{x \geq C_2 n^q} ,
	\end{split}
	\end{equation}
	for $C_2$ and $q$ constants whose values are chosen later. For ${\bf Y}:= (y_1,...,y_m) \in \left(\Z^d\right)^{\otimes m}$, define the region
	\begin{equation}
	J_{\bf Y} := \{(kn + i, \sqrt{n} y_k + z)\in \Z^{1+d} : k = 0, . . . , m - 1, i = 1, . . . , n, |z_j| \leq C_3 \sqrt{n} \quad \forall j\leq d\},
	\end{equation}
	where $C_3$ is a positive constant. This region is defined to take advantage of the concentration properties of the simple random walk. By Holder's Inequality,
	\begin{equation} \label{holder}
	\begin{split}
	\E   \left(\tilde{Z}_{y_1,...,y_m} \right)^{\theta} 
	&= \E \left[ \prod_{(i,z)\in J_{\bf Y}} g(\w_{i,z})^{-(1-\theta)} \prod_{(i,z)\in J_{\bf Y}} g(\w_{i,z})^{1-\theta}  \left(\tilde{Z}_{y_1,...,y_m} \right)^{\theta} \right]\\
	&\leq  \E \left[ \prod_{(i,z)\in J_{\bf Y}} g(\w_{i,z})^{-1} \right]^{1-\theta} \E \left[ \prod_{(i,z)\in J_{\bf Y}} g(\w_{i,z})^{\frac{1-\theta}{\theta}}   \left(\tilde{Z}_{y_1,...,y_m} \right) \right]^{\theta},
	\end{split}
	\end{equation}
	Notice that $|J_{\bf Y}|= n m (2C_3 \sqrt{n})^d$. Hence, for the first factor, we have
	\begin{equation}
	\begin{split} \label{cost}
	\E \left[ \prod_{(i,z)\in J_{\bf Y}} g(\w_{i,z})^{-1} \right] &= \E \left[ g(\w_{0,0})^{-1} \right]^{|J_{\bf Y}|} \\
	&\leq \left(1 + 2\P\left[\w_{1,0} > C_2 n^q \right]   \right)^{2^{d}C_3^d m  n^{1+d/2}}\\
	&\leq \exp\left(2^{d+1} C'' m (C_2 n^q)^{-\g} C_3^d n^{1+d/2}  \right)\\
	&= \exp\left(2^{d+1} C'' m \right),
	\end{split} 
	\end{equation}
	by choosing $C_2 = C_3^{d/\g}$ and $q = \frac{2+d}{2\g}$. Notice that for the second inequality we used 
	\begin{equation}
	C' x^{-\g} \leq \P[\w_{1,0} > x ] \leq C'' x^{-\g},  
	\end{equation}
	for some constants $C', C''$ with $C'< C_\P < C''$, for all $x$ sufficiently large, since \eqref{condition2}.
	By Fubini's Theorem, we have for the second factor,
	\begin{equation}
	\begin{split}
	\E \left[ \prod_{(i,z)\in J_Y} g(\w_{i,z})^{\frac{1-\theta}{\theta}}   \left(\tilde{Z}_{y_1,...,y_m} \right) \right] &= 
	\E \left[ \prod_{(i,z)\in J_Y} g(\w_{i,z})^{\frac{1-\theta}{\theta}}
	\bE \left[ \prod_{i=1}^{n m}
	(1 + \beta \omega_{i,S_i})
	\prod_{k=1}^{m} \1{S_{kn} \in I_{y_k}} \right] \right] \\
	&= \bE \left[ \E \left[ \prod_{(i,z)\in J_Y} g(\w_{i,z})^{\frac{1-\theta}{\theta}}
	\prod_{i=1}^{n m}
	(1 + \beta \omega_{i,S_i})
	\prod_{k=1}^{m} \1{S_{kn} \in I_{y_k}} \right] \right] \\
	&\leq \bE \left[ 
	\prod_{i \in \mathcal{I}\left(S,J_Y \right)  }\E \left[  (1 + \beta \omega_{1,0}) g(\w_{1,0})^{\frac{1-\theta}{\theta}}
	\right]\prod_{k=1}^{m} \1{S_{kn} \in I_{y_k}} \right],
	\end{split}  
	\end{equation}
	where for a given walk $S$ and a finite subset $J \subset \Z^{1+d}$ we define 
	\begin{equation}
	\mathcal{I}\left(S,J \right):= \left\lbrace i \in \N :(i,S_i) \in J \right\rbrace.
	\end{equation}
	In the last inequality above, we neglect sites on $J_Y$, for which the paths do not visit, since $\E[g(\w_{(1,0)})]\leq 1$ and sites for which paths do visit but are outside $J_Y$, since $\E[1 + \b \w_{1,0}] = 1$.
	A simple computation shows that, for $\theta = 1/2$,
	\begin{equation}
	\begin{split}
	\E \left[  (1 + \beta \omega_{1,0}) g(\w_{1,0})^{\frac{1-\theta}{\theta}}
	\right] &= \E \left[  (1 + \beta \omega_{1,0}) \left( 1 - \tfrac{1}{2} \1{\w_{1,0} \geq C_2 n^q} \right)
	\right]  \\
	&\leq \exp(-\tfrac{1}{2}C' \b (C_2 n^q)^{-(\g-1)}).
	\end{split}
	\end{equation}
	Using the Markov Property, we obtain
	\begin{equation} \label{change of measure}
	\begin{split}
	&\E \left[ \prod_{(i,z)\in J_Y} g(\w_{i,z})^{\frac{1-\theta}{\theta}}  \left(\tilde{Z}_{y_1,...,y_m} \right) \right] \leq \bE \left[ 
	\exp(-\tfrac{1}{2}C' \b (C_2 n^q)^{-(\g-1)}  \left| \mathcal{I}\left(S,J_Y \right)\right| ) 
	\prod_{k=1}^{m} \1{S_{kn} \in I_{y_k}} \right] \\
	&\leq \prod_{k=1}^{m} \max_{x\in I_0} \bE_x \left[ 
	\exp(-\tfrac{1}{2}C' \b (C_2 n^q)^{-(\g-1)} | \mathcal{I} (S,\tilde{J} )|) \1{S_{n} \in I_{y_k - y_{k-1}}} \right],
	\end{split}
	\end{equation}
	where we define the set $\tilde{J}$ as
	\begin{equation}
	\tilde{J}:= 
	\{(i,z)\in \Z^{1+d} : i = 1, . . . , n, |z_j| \leq C_3 \sqrt{n} \quad \forall j\leq d\}.
	\end{equation}
	Combining \eqref{coarse} \eqref{holder}, \eqref{cost} and \eqref{change of measure} we obtain
	\begin{equation}
	\begin{split}
	\log \E  \left[  \left(\tilde{Z}_{y_1,...,y_m} \right)^{\theta} \right] &\leq m  2^{d} C'' +  m \log \sum_{y \in \Z^d} \max_{x\in I_0} \bE_x \left[ 
	e^{-\tfrac{1}{2}C' \b (C_2 n^q)^{-(\g-1)} | \mathcal{I} (S,\tilde{J} )|}\1{S_{n} \in I_y}   \right]^{1/2}.
	\end{split}
	\end{equation}
	Then, by \eqref{goal}, it is enough to show that the expression 
	\begin{equation}
	\sum_{y \in \Z^d} \max_{x\in I_0} \bE_x \left[ 
	e^{-\tfrac{1}{2}C' \b (C_2 n^q)^{-(\g-1)} | \mathcal{I} (S,\tilde{J} )|} \1{S_{n} \in I_y}   \right]^{1/2},
	\end{equation}
	is sufficiently small. For values of $y$ far form the origin, we neglect the contribution of the change of measure:
	\begin{equation} \label{above}
	\sum_{|y|_\infty > C_4} \max_{x\in I_0} \bE_x \left[ 
		e^{-\tfrac{1}{2}C' \b (C_2 n^q)^{-(\g-1)} | \mathcal{I} (S,\tilde{J} )|} \1{S_{n} \in I_y}   \right]^{1/2} \leq
	\sum_{|y|_\infty > C_4} \max_{x\in I_0} \bP_x \left[ 
	S_{n} \in I_y  \right]^{1/2}.
	\end{equation}
	Applying standard results on sums of IID random variables, we can bound
	\begin{equation}
	\max_{x\in I_0} \bP_x \left[ 
	S_{n} \in I_y  \right] \leq \bP \left[ 
	|S_{n}|_\infty \geq  (|y|_\infty - 1)\sqrt{n}  \right] \leq e^{-c |y|_\infty^2},
	\end{equation}
	for a fixed constant $c > 0 $. In this manner, we make the sum \eqref{above} arbitrarily small, by choosing $C_4$ large enough.
	For values of $y$ near from the origin, we neglect the condition over $S_n$:
	\begin{equation}
	\begin{split}
	\sum_{|y|_\infty \leq R} \max_{x\in I_0} \bE_x \left[		e^{-\tfrac{1}{2}C' \b (C_2 n^q)^{-(\g-1)} | \mathcal{I} (S,\tilde{J} )|} \1{S_{n} \in I_y}   \right]^{1/2}\\
	\leq (2R)^{d} \max_{x\in I_0} \bE_x \left[		e^{-\tfrac{1}{2}C' \b (C_2 n^q)^{-(\g-1)} | \mathcal{I} (S,\tilde{J} )|}   \right]^{1/2}\\
	\leq (2R)^{d} \bE \left[		e^{-\tfrac{1}{2}C' \b (C_2 n^q)^{-(\g-1)} | \mathcal{I} (S,\bar{J} )|}   \right]^{1/2}, \\
	\end{split}
	\end{equation}
	where the set $\bar{J}$ is defined as
	\begin{equation}
	\bar{J}:= 
	\{(i,z)\in \Z^{1+d} : i = 1, . . . , n, |z_j| \leq (C_3 - 1 ) \sqrt{n} \quad \forall j\leq d\}.
	\end{equation}
	In the last line, we use the fact that for any walk S, starting at zero and $x \in I_0$,
	\begin{equation}
	\{i:(i,S_i) \in \bar{J}\} \subset \{i:(i,x + S_i) \in \tilde{J}\}.
	\end{equation}
	The last expression can be bounded as
	\begin{equation}
\bE \left[		e^{-\tfrac{1}{2}C' \b (C_2 n^q)^{-(\g-1)} | \mathcal{I} (S,\tilde{J} )|}   \right] \leq
	\bP\left[\exists i: (i,S_i) \notin \bar{J} \right] + 
	e^{-\tfrac{1}{2}C' \b (C_2 n^q)^{-(\g-1)} n}.
	\end{equation}
	The first term in the last sum can be made arbitrarily small, by choosing $C_3$ sufficiently big. For the second term, using our initial assumption: $n > \frac{C_1}{2 \b^\a}$ and the values of $\a = \frac{\g(\g_c-1)}{\g_c - \g}$, $q = \frac{2 + d}{2\g}$ and $\g_c = 1+\frac{2}{d}$,  we obtain,
	\begin{equation} \label{alphaa}
	\begin{split}
	e^{-\tfrac{1}{2}C' \b (C_2 n^q)^{-(\g-1)}n}&\leq e^{-\tfrac{1}{2}C' \b C_2^{-(\g-1)} n^{1/\a}}  \leq e^{-\tfrac{1}{2}C'  C_2^{-(\g-1)} \left( \frac{C_1}{2}\right)^{1/\a} },
	\end{split}
	\end{equation}    
	which can also be made arbitrarily small, by choosing $C_1$ sufficiently big. This completes the proof of the theorem.
\end{proof}
\begin{rem}\label{rem22}
The value of $\g_c= 1 + \frac{2}{d}$ appears naturally in the computations. Notice that in \eqref{alphaa}, for the value of $\a$ to be positive we need $1 - q(\g-1) > 0$ which implies $\g < 1 + 2/d$.
\end{rem}
\section{Marginal Case.}
In this section we show an upper bound for the free energy for the case $\g = \g_c$. This completes the proof that very strong disorder holds for all values of $\b > 0$, when $\g \leq \g_c$. The first part of the proof shares steps from the case when  $\g < \g_c$ up to the point of choosing the change of measure function, since that is no longer suitable in this case (cf. Remark \ref{rem22}).  The approach we take here is to penalize regions of the environment that contains a pair of sites whose values are above a certain threshold that depends on the distance between each other. The idea behind this is to account the fact that pairs of sites that are more close to each other produces a more noticeable effect on the partition function. The construction of this change of measure is inspired by the one used in \cite{marginal} to prove disorder relevance for the pinning model.
\begin{proposition}
	When $\g = \g_c$, and assuming the usual hypotesis on the environment's distribution, we have
	\begin{equation}
	p(\b) \leq - C \exp\left( - \frac{c}{\b^{2\g}}\right),
	\end{equation}
	for all $\b > 0$ sufficiently small, and some fixed constants $C, c>0$. 
\end{proposition}
\begin{proof}
	Let $n \in \N$ be a squared natural number satisfying
	\begin{equation} \label{feq}
	\exp\left(\frac{C_5}{\b^{2\g}}\right) \leq n \leq \exp\left(\frac{2 C_5}{\b^{2\g}}\right).  
	\end{equation}
	where $C_5$ is a constant to be chosen later. As before, our goal is to show,
	\begin{equation}
	\E\left[ \left( Z_{mn}^{\b, \w}\right)^\theta \right] \leq \exp\left(-m \right),
	\end{equation} 
	for all $m \in \N$ and some $\theta \in (0,1)$. Using the notation from the previous section,
	\begin{eqnarray}\label{jen2}
	\E  \left[ \left( Z_{mn}^{\beta,\omega}\right) ^\theta \right]  \leq \sum_{y_1,...,y_m \in \Z^d} \E\left[  \left(  \tilde{Z}_{y_1,...,y_m} \right)^\theta \right] ,
	\end{eqnarray}
	where $\tilde{Z}_{y_1,...,y_m} = \bE \left[ \prod_{i=1}^{n m}
	(1 + \beta \omega_{(i,S_i)})
	\prod_{k=1}^{m} \1{S_{kn} \in I_{y_k}} \right]$ for  $y_1,...,y_m \in \Z^d$. Fixing the collection of vertex ${\bf Y} := (y_1,...,y_m)$ we define the blocks $\left(B_k\right)_{1\leq k \leq m}$ as
	\begin{equation}
	B_k := \left\lbrace (i,z) \in \N \times \Z^d : \lceil  i/n \rceil=k, \left|z - \sqrt{n}y_{k-1} \right|_\infty \leq C_6 \sqrt{n} \right\rbrace, 
	\end{equation} 
	where $C_6 > 0$ is chosen later. We also define the set of functions $\left\lbrace g_k(\w)\right\rbrace_{1\leq k \leq m} $ as
	\begin{equation}
	g_k(\w) := \exp\left( -M \text{\Large 1}_{A_k}\right), 
	\end{equation}
	where
	\begin{equation}
	A_k := \left\lbrace \exists (i,z), (j,z') \in B_k : |z-z'|\leq C_7 \sqrt{|i-j|}, \w_{i,z} \wedge \w_{j,z'}   \geq V(\left| i-j\right| )  \right\rbrace,
	\end{equation}
	and
	\begin{equation}
	V(t ) := 
	\exp(M^2)\left( C_6^{d} C_7^{d} n^{1+\frac{d}{2}} t^{1+\frac{d}{2}} \log n\right)^{\frac{1}{2 \gamma}},
	\end{equation}
	for $t > 0$ and $V(0) := \infty$. 
	Now we can define our change of measure function $G_{\bf Y}(\w)$ as
	\begin{equation}
	G_{\bf Y}(\w) := \prod_{k=1}^{m} g_k(\w).
	\end{equation}
	Apply Hölder's Inequality to obtain
	\begin{equation}
	\E\left[  \left(  \tilde{Z}_{y_1,...,y_m} \right)^\theta \right] \leq \E \left[G_I^{-\frac{\theta}{1-\theta}}
	\right]^{1-\theta} \E\left[ G_I  \tilde{Z}_{y_1,...,y_m} \right]^\theta. 
	\end{equation}
	Notice first that
	\begin{equation}
	\E \left[g_k(\w)^{-\frac{\theta}{1-\theta}}
	\right] = \E \left[g_k(\w)^{-\frac{\theta}{1-\theta}}
	\text{\Large 1}_{A_k} \right] + \E \left[g_k(\w)^{-\frac{\theta}{1-\theta}} \text{\Large 1}_{A_k^c}
	\right] \leq \exp\left(\frac{M \theta}{1-\theta}\right) \P \left[ A_k \right] + 1.
	\end{equation}
	The probabilities of the events $A_k$ are sufficiently small so that the product above is well controlled,
	\begin{align}
	\P \left[ A_k \right] &\leq 2 \sum_{(i,z) \in B_k} \sum_{t=1}^n \sum_{ \substack{z' \in \Z^d \\ |z-z'|\leq C_7 \sqrt{t}}}
	\P\left[\w_{i,z} \wedge \w_{i+t,z'}  \geq V(t) ,  \right] \\
	&\leq 2 \left( 2C_6\sqrt{n}\right)^d n  \sum_{t=1}^n \left( 2C_7\sqrt{t}\right)^d  C_\P^2 \left( \exp(M^2)\left(C_6^{d} C_7^{d} n^{1+\frac{d}{2}} t^{1+\frac{d}{2}} \log n\right)^{\frac{1}{2 \gamma}}\right)^{-2\g} \\
	&= 2^{2d+1} C_\P^2 \exp\left(-2\g M^2 \right) \frac{1}{\log n}  \sum_{t=1}^n \frac{1}{t},
	\end{align}
	and we can choose $M$ sufficiently large, such that $ 
	\E \left[g_k(\w)^{-\frac{\theta}{1-\theta}}
	\right] \leq 2.$
	Then,
	\begin{equation} \label{2m}
	\E \left[G_I^{-\frac{\theta}{1-\theta}}
	\right] \leq 2^m.
	\end{equation}
	Now we are left with the estimation of the second term. As before, we have
	\begin{align}\label{max}
	\E\left[ G_I \tilde{Z}_{y_1,...,y_m} \right] \leq \prod_{k=1}^{m}  \max_{x \in I_0} \bE_x \left[ \E \left[ g_1(\w) \prod_{i=1}^{n}
	(1 + \beta \omega_{i,S_i}) \right] 
	\1{S_{n} \in I_{y_k - y_{k-1}}} \right].
	\end{align}
	Plugging \eqref{2m} and \eqref{max} in Equation \eqref{jen2}, we obtain that
	\begin{equation}
	\E  \left[ \left( Z_{mn}^{\beta,\omega}\right) ^\theta \right] \leq 2^{m(1-\theta)} \left(\sum_{y\in \Z^d} \max_{x \in I_0} \bE_x \left[ \E \left[ g_1(\w) \prod_{i=1}^{n}
	(1 + \beta \omega_{i,S_i}) \right] 
	\1{S_{n} \in I_{y}} \right]^{\theta} \right)^{m},
	\end{equation}
	so it will be sufficient to show that 
	\begin{equation}
	\sum_{y\in \Z^d} \max_{x \in I_0} \bE^x \left[ \E \left[ g_1(\w) \prod_{i=1}^{n}
	(1 + \beta \omega_{(i,S_i)}) \right] 
	\1{S_{n} \in I_{y}} \right]^{\theta}
	\end{equation}
	is small. The contribution of $y$ far from the origin can be controlled as in Equation \eqref{above}. Thus it is sufficient to check that
	\begin{equation}
	\max_{x \in I_0} \bE_x \left[ \E \left[ g_1(\w) \prod_{i=1}^{n}
	(1 + \beta \omega_{i,S_i}) \right] 
	\1{S_{n} \in I_{y}} \right]^{\theta} \leq \eps,
	\end{equation} 
	for some arbitrarily small $\eps > 0$. We choose $C_6$ sufficiently big, such that
	\begin{equation}
	\bP \left[ |S_{i}| > (C_6 - 1)\sqrt{n},\text{   for some }i\leq n \right] \leq \frac{\eps}{2}.
	\end{equation}
	Then it is enough to prove,
	\begin{equation}
	\bE \left[ \E \left[ g_1(\w) \prod_{i=1}^{n}
	(1 + \beta \omega_{i,S_i}) \right] 
	\1{S_{i} \in B_1, \text{   for all }i\leq n} \right] \leq \frac{\eps}{2}.
	\end{equation}
	Notice that when $S$ is fixed, the change of measure induced by the function $\prod_{i=1}^n \left(1+\b \w_{i,S_i} \right)$ retains the independence of the elements of the environment but tilts the distribution of the ones that belong to the graph of $S$ by a factor of $\left(1+\b \w_{i,S_i} \right)$. This allows us to consider an IID set of random variables $\tilde{\w} = \{\tilde{\w}_{n,z}: n \in \N,z \in \Z^{d}\}$ from a probability space $ (\tilde{\Lambda}, \tilde{\cF}, \tilde{\P})$ of distribution given by:
	\begin{equation} \label{quenched}
	\tilde{\P}\left( \tilde{\w}_{1,0} \in \cdot \right) = \E\left[ \left(1 + \b \w_{1,0}\right) \1{\w_{1,0} \in \cdot} \right],  
	\end{equation}
	and express the inequality above as
	\begin{equation} \label{want}
	\bE \left[ \E \otimes \tilde{\E} \left[ g_1(\hat{\w}^S)  \right] 
	\1{S_{i} \in B_1, \text{   for all }i\leq n} \right] \leq \frac{\eps}{2}.
	\end{equation}	
	where for all $i \in \N$ and $z \in \Z^{d}$ we define $\hat{\w}_{(i,z)}^S$ as:
	\begin{equation}
	\hat{\w}_{(i,z)}^S := \w_{(i,z)}\1{z \not= S_i} + \tilde{\w}_{(i,z)}\1{z = S_i}.
	\end{equation}
	The idea for the rest of the proof is to show that, under the measure $\P \otimes \tilde{\P}$, the event $\{\hat{\w}^S \in  A_1\}$ is very likely, so $g_1(\hat{\w}^S)$ is equal to $\exp(-M)$ with high probability. Then, taking $M$ large should be sufficient.
	The following estimates holds for the distributions of the tilted environment, where $C>0$ is some constant and $x \geq \b^{-1}$:
	\begin{equation}
	\frac{1}{C}  \b x^{-\g +1} \leq \tilde{\P} \left[ \tilde{\w}_{1,0} \geq x \right] \leq C \b x^{-\g +1}.
	\end{equation}
	Define the random variable
	\begin{equation}
	X(\hat{\w}^S) := \sum_{0 \leq i,j < n} \1{|S_i-S_j|\leq C_7 \sqrt{|i-j|},  \hat{\w}^S_{i,S_i} \wedge \hat{\w}^S_{j,S_j}  \geq V(|i-j|)}.
	\end{equation}
	Notice that $X(\hat{\w}^S) \geq 1$ implies that $\hat{\w}^S \in A_1$ and that we can lower bound the expectation of $X(\hat{\w}^S)$ under $\P \otimes \tilde{\P}$ by,
	\begin{align}
	\E \otimes \tilde{\E} \left[ X(\hat{\w}^S) \right] &\geq  \sum_{i=0}^{n/2} \sum_{t=1}^{n/2} \tilde{\P} \left[\tilde{\w}_{1,0} \geq V(t) \right]^2 \1{|S_i-S_{i+t}|\leq C_7 \sqrt{t}}.
	\end{align}
	Since 
	\begin{equation}
	\sum_{i=0}^{n/2} \sum_{t=1}^{n/2} \tilde{\P} \left[\tilde{\w}_{1,0} \geq V(t) \right]^2 \1{|S_i-S_{i+t}|\leq C_7 \sqrt{t}} \leq \sum_{i=0}^{n/2} \sum_{t=1}^{n/2} \tilde{\P} \left[\tilde{\w}_{1,0} \geq V(t) \right]^2,
	\end{equation}
	we can choose $C_7$ sufficiently big, such that, by Markov inequality,
	\begin{align} \label{ineq12}
\bP \left[	\sum_{i=0}^{n/2} \sum_{t=1}^{n/2} \tilde{\P} \left[\tilde{\w}_{1,0} \geq V(t) \right]^2 \1{|S_i-S_{i+t}|\leq C_7 \sqrt{t}} \leq \frac{1}{2}\sum_{i=0}^{n/2} \sum_{t=1}^{n/2} \tilde{\P} \left[\tilde{\w}_{1,0} \geq V(t) \right]^2\right] \leq \eps/4.
\end{align}	
Then,	
\begin{align} \label{ineq1}
	\bP \left[\E \otimes \tilde{\E} \left[ X(\hat{\w}^S) \right] \geq \frac{1}{2}    \sum_{i=0}^{n/2} \sum_{t=1}^{n/2} \tilde{\P} \left[\tilde{\w}_{(1,0)} \geq V(t) \right]^2 \right] \geq 1- \eps/4.
	\end{align}
	In the same event, we have that
	\begin{align}
	\E \otimes \tilde{\E} \left[ X(\hat{\w}^S) \right] &\geq \frac{1}{2} \sum_{i=0}^{n/2} \sum_{t=1}^{n/2} \tilde{\P} \left[\tilde{\w}_{1,0} \geq V(t) \right]^2 \geq \frac{n}{4} \sum_{t=1}^{n/2} \left( C^{-1} \b V(t)^{-\g + 1} \right)^2 \\
	&= \frac{n}{4} \sum_{t=1}^{n/2} \left( C^{-1} \b \left(\exp(M^2)\left( C_6^{d} C_7^{d} t^{1+\frac{d}{2}} n^{1+\frac{d}{2}} \log n\right)^{\frac{1}{2 \gamma}} \right)^{-\g + 1} \right)^2 \\
	&\geq C' \b^2 \left(\log n \right)^{\frac{1}{\g}}.  
	\end{align}
	for a constant $C'$ that might depend on $M,C_6$ and $C_7$, whose values have already been chosen. On the other hand, after canceling all terms with covariance zero, we can bound the variance of $X$ as
	\begin{equation}
	\begin{split}
	\text{Var}_{\P \otimes \tilde{\P}} &\left[ X(\hat{\w}_S) \right] \leq \sum_{0 \leq i,j < n} \tilde{\P} \left[  \tilde{\w}_{i,S_i} \wedge \tilde{w}_{j,S_j}  \geq V(|i-j|) \right] \\
	&+ 4\sum_{0 \leq i,j,k < n} \tilde{\P} \left[ \tilde{w}_{i,S_i}\wedge \tilde{w}_{j,S_j}  \geq V(|i-j|),  \tilde{w}_{i,S_i} \wedge \tilde{w}_{k,S_k}  \geq V(|i-k|) \right].
	\end{split}
	\end{equation}
	The first term in the sum is similar to the expectation of $X$, and by an analogous computation, we have that,
	\begin{equation}
	\E \otimes \tilde{\E}[X(\hat{\w}_S)] \leq C'' \b^2 \left(\log n \right)^{1/\g}.
	\end{equation} 
	Let us called $Y$ the second term in the sum. Rearranging the terms of $Y$, we have that,
	\begin{equation}
	\begin{split}
	Y &\leq 32 \sum_{i = 0}^{n-1} \sum_{t = 1}^{n} \sum_{t' = 1}^{t} \tilde{\P} \left[  \tilde{w}_{i,S_i}\wedge \tilde{w}_{(i+t,S_{i+t})}  \geq V(t),  \tilde{w}_{i,S_i} \wedge \tilde{w}_{i+t',S_{i+t'}}  \geq V(t') \right] \\
	& \leq 32 \sum_{i = 0}^{n-1} \sum_{t = 1}^{n} \sum_{t' = 1}^{t} \tilde{\P} \left[\tilde{w}_{i,S_i} \geq V(t)\right] \tilde{\P} \left[ \tilde{w}_{i+t',S_{i+t'}} \geq V(t') \right] \tilde{\P} \left[ \tilde{w}_{i+t,S_{i+t}} \geq V(t) \right] \\
	&\leq 32 \sum_{i = 0}^{n-1} \sum_{t = 1}^{n} \sum_{t' = 1}^{n} \P \left[w_{1,0} \geq V(t)\right]^2 \P \left[ w_{1,0} \geq V(t') \right] \\
	&\leq C''' n \sum_{t = 1}^{n}\b \left(\left(  n^{1+\frac{d}{2}} t^{1+\frac{d}{2}} \log n\right)^{\frac{1}{2 \gamma}} \right)^{2(-\g + 1)} \sum_{t' = 1}^{n} \b \left(\left(n^{1+\frac{d}{2}} t'^{1+\frac{d}{2}}\log n\right)^{\frac{1}{2 \gamma}} \right)^{-\g + 1}  \\
	&\leq C''' \b^3 \left( \log n\right)^{\frac{3}{2} \frac{1}{\g}}. 
	\end{split}
	\end{equation}
	By Chebychev's Inequality, in the event $\{ \E \otimes \tilde{\E} \left[ X(\hat{\w}_S) \right] \geq \frac{1}{2}    \sum_{i=0}^{n/2} \sum_{t=1}^{n/2} \tilde{\P} \left[\tilde{w}_{1,0} \geq V(t) \right]^2 \}$, we have that
	\begin{equation}
	\P \otimes \tilde{\P} \left[ X(\hat{\w}_S) = 0 \right] \leq \frac{\text{Var}_{\P \otimes \tilde{\P}}\left[ X(\hat{\w}_S) \right]}{\E \otimes \tilde{\E} \left[ X(\hat{\w}_S) \right]^2} \leq \frac{C'' \b^2 \left(\log n \right)^{1/\g} + C''' \b^3 \left( \log n\right)^{\frac{3}{2} \frac{1}{\g}}}{\left(C' \b^2 \left(\log n \right)^{\frac{1}{\g}}\right)^2}.
	\end{equation}
	Recall Equation \eqref{feq} and choose $C_5$ large enough such that the last term is smaller than $\eps / 8$ and $M$ such that $\exp(-M)\leq \eps / 8$. Using this and \eqref{ineq1} we obtain \eqref{want}.
\end{proof}

\section{Disorder irrelevance.}
In this section we are going to prove Theorem \ref{2nd}. The idea is to show that when $\g > 1 + 2/d$, the sequence of partition functions is uniformly integrable so that $\E\left[  Z_\infty \right] = 1$. This is perform by bounding the $(1+q)$-th moment of the partition function for some positive $q$. In order to do this, we rewrite the problem as the estimation of the $q$-th moment of the partition function of the system where the environment's distribution has been tilted along a quenched path. This technique has appeared originally in \cite{Lacoin} for the pinning model case.
\begin{proposition}
		When $\g > \g_c$, $d \geq 3$ and assuming the usual hypothesis on the environment's distribution, we have
			\begin{equation} \label{ui}
		\sup_{N\in \N} \E \left[ \left( Z_N^{\b, \w} \right)^{1+q} \right] < \infty,
		\end{equation}
for all  $\b > 0$ sufficiently small, and some $q \in (0,\g -1)$.
\end{proposition}
\begin{proof}
	Rewrite partition function above as
	\begin{equation}
	\begin{split}
	\E \left[ \left( Z_N^{\b, \w} \right)^{1+q} \right] &=  \E \left[\bE \left[\prod_{i=1}^N \left(1+\b \w_{i,S_i} \right)  \right]  \left( Z_N^{\b, \w} \right)^{q} \right] \\
	&= \bE \left[\E \left[\left( Z_N^{\b, \w} \right)^{q} \prod_{i=1}^N \left(1+\b \w_{i,S_i} \right)   \right]   \right].
	\end{split} 
	\end{equation}
	Using the notation introduced in Equation \eqref{quenched} we write the expectation above as
	\begin{equation}
	\E \left[\left( Z_N^{\b, \w} \right)^{q} \prod_{i=1}^N \left(1+\b \w_{i,S_i} \right)   \right]  = \E \otimes \tilde{\E} \left[\left( Z_N^{\b, \hat{\w}^S} \right)^{q}  \right],  
	\end{equation}
	where
	\begin{equation}
	\hat{\w}_{i,z}^S := \w_{i,z}\1{z \not= S_i} + \tilde{\w}_{i,z}\1{z = S_i},
	\end{equation}
	for all $i \in \N, z \in \Z^{d}$. The reason for which this consideration might be useful is that more techniques
	are available to control $p$-moments for $p$ in the interval $(0, 1)$ than for $p > 1$. We can also express $Z_N^{\b, \hat{\w}^S}$ as 
	\begin{equation}
	Z_N^{\b, \hat{\w}^S} = \bE' \prod_{i=1}^{N} \left(1+\b \hat{\w}_{i,S'_i}^S \right) ,
	\end{equation}
	where $  \left( \O',\bP', S'\right) $ is an independent copy of $  \left( \O,\bP, S\right) $. By Fubini's Identity and Jensen's Inequality
	\begin{equation}
	\begin{split}
	\E \left[ \left( Z_N^{\b, \w} \right)^{1+q} \right] &\leq  \bE \otimes \tilde{\E}\otimes \E \left[\left(\bE' \prod_{i=1}^{N} \left(1+\b \hat{\w}_{i,S'_i}^S \right) \right)^{q} \right]\\
	&\leq \bE \otimes \tilde{\E}\left[\left( \bE'  \prod_{i=1}^N \left(1+\b \tilde{\w}_{i,S'_i}\1{S_i = S'_i} \right)  \right)^q \right]. \\	\end{split}
	\end{equation}
	We used Jensen inequality here to obtain a more tractable expression to
	estimate. Notice that we cannot simply
	apply it for $\tilde{\E}$ as $\tilde{\w}$ has infinite mean.
Also notice that, we can rewrite the expectation inside as
	\begin{equation}
	\tilde{\E}\left[\bE' \left[ \prod_{i=1}^N \left(1+\b \tilde{\w}_{i,0}\1{S_i = S'_i} \right) \right]^q    \right], 
	\end{equation} 
	since for a fixed path $S$, the joint distributions of $\left\lbrace \tilde{\w}_{(1,S_1)}, \dots, \tilde{\w}_{(N,S_N)} \right\rbrace $ and
	$\left\lbrace \tilde{\w}_{(1,0)}, \dots, \tilde{\w}_{(N,0)} \right\rbrace $ are identical. By simplicity we write $\tilde{\w}_{i,0}$ as $\tilde{\w}_i$. Using Jensen's Inequality and Fubini's Theorem one more time, we have
	\begin{equation}
	\E \left[ \left( Z_N^{\b, \w} \right)^{1+q} \right] \leq
	\tilde{\E} \left[ \left( \bE'\otimes \bE \prod_{i=1}^N \left(1+\b \tilde{\w}_i \1{S_i = S'_i} \right) \right) ^q    \right].
	\end{equation}
	Observe that the expression on the right corresponds to the $q$-th moment for the partition function of a one dimensional pinning model, with free boundary condition, associated with a transient renewal process $\overline{\tau}$ whose inter-arrival distribution $\overline{\bP}$ satisfies:
	\begin{equation} \label{renewal}
	\overline{\bP}\left[ \overline{\tau} = n\right] =: K(n) = \bP \otimes \bP' \left[S_1 \not= S_1',\dots,S_{n-1} \not= S'_{n-1}, S_n = S_n' \right],
	\end{equation}
	and an environment given by a realization of $\left\lbrace \tilde{\w}_1, \dots, \tilde{\w}_N \right\rbrace $. With this notation we can write
	\begin{equation}
	\bE'\otimes \bE \prod_{i=1}^N \left(1+\b \tilde{\w}_i \1{S_i = S'_i} \right)  = \overline{\bE} \prod_{i=1}^N \left( 1 + \b \tilde{\w}_i \1{i \in \overline{\tau}} \right).
	\end{equation}
	Let us consider the constrained version of the pinning model:
	\begin{equation}
	\bar{Z}_N^{\b, \tilde{\w}} := \overline{\bE}\left[\prod_{i=1}^N \left(1+\b \tilde{\w}_i \1{ i \in \overline{\tau}} \right) \1{N \in \overline{\tau}} \right].
	\end{equation}
	Notice that it is sufficient to show that
	\begin{equation} \label{bound on pinning}
	\sum_{N=1}^\infty \tilde{\E} \left[ \left( \bar{Z}_N^{\b, \tilde{\w}}   \right)^q \right] < \infty,
	\end{equation}
	since we can write, by the Markov property,
	\begin{equation} \label{pinning}
	\begin{split}
	\overline{\bE} \left[ \prod_{i=1}^N \left( 1 + \b \tilde{\w}_i \1{i \in \overline{\tau}} \right) \right] &= \sum_{j=0}^N \overline{\bE} \left[ \prod_{i=1}^N \left( 1 + \b \tilde{\w}_i \1{i \in \bar{\tau}} \right) \1{j \in \bar{\tau},j+1 \notin \bar{\tau},\dots,N \notin \bar{\tau}} \right]\\
	&=\sum_{j=1}^N \bar{Z}_j^{\b, \tilde{\w}} \overline{\bP} \left[ \bar{\tau} > N-j \right] \leq  \sum_{j=1}^N \bar{Z}_j^{\b, \tilde{\w}},
	\end{split}
	\end{equation}
	and 
	\begin{equation}
	\tilde{\E}\left[\left(\sum_{j=1}^N \bar{Z}_j^{\b, \tilde{\w}} \right)^q  \right] \leq  \sum_{N=1}^\infty \tilde{\E} \left[ \left( \bar{Z}_N^{\b, \tilde{\w}}   \right)^q \right].
	\end{equation}
	Similarly as we did in the previous sections, we apply Holder's Inequality and a change of measure to get ride of the exponent $q$. By H\"older's Inequality, we have
	\begin{equation}
	\tilde{\E} \left[ \left( \bar{Z}_N^{\b, \tilde{\w}}   \right)^q \right] \leq 
	\tilde{\E} \left[ \left(  \prod_{i=1}^{N}h(\tilde{\w}_i )^{-1} \right)^{\frac{1}{1-q}} \right]^{1-q} 
	\tilde{\E} \left[ \left(  \prod_{i=1}^{N}h( \tilde{\w}_i) \right)^{1/q} \bar{Z}_N^{\b, \tilde{\w}} \right]^{q},
	\end{equation}
	for some positive function $h:\R \to \R$.  We choose to use
	\begin{equation}
	\begin{split}
	h:\R &\to \R \\
	x &\mapsto (1+\b x)^{-q(1-q)},
	\end{split}
	\end{equation}
	for the change of measure as in \cite{Lacoin} since it gives us the same weight $ \tilde{\E} \left[ (1 + \b \tilde{\w}_i)^q \right]$ for both factors after applying H\"older's Inequality, as it is seen below in Equations \eqref{key11} and \eqref{key22}. We compute the first expectation using the IID structure of the environment:
	\begin{equation}\label{key11}
	\tilde{\E} \left[ \left(  \prod_{i=1}^{N}h(\tilde{\w}_i)^{-1} \right)^{\frac{1}{1-q}} \right] = \tilde{\E} \left[ (1 + \b \tilde{\w}_i)^q \right]^{N} = \E\left[(1 + \b \w_i)^{1+q} 
	\right]^{N}.
	\end{equation}
	For the second expectation,
	\begin{equation} \label{key22}
	\begin{split}
	\tilde{\E} \left[ \left(  \prod_{i=1}^{N}h( \tilde{\w}_i) \right)^{1/q} \bar{Z}_N^{\b, \tilde{\w}} \right] &= 
	\tilde{\E} \left[ \left(  \prod_{i=1}^{N} (1+\b  \tilde{\w}_i) \right)^{-(1-q)} 
	\overline{\bE}\left[\prod_{i=1}^N \left(1+\b \tilde{\w}_i \1{ i \in \overline{\tau}} \right) \1{N \in \overline{\tau}} \right]
	\right] \\ 
	&\leq \overline{\bE}\left[\prod_{i \in \overline{\tau} \cap [1,N]} 
	\E\left[(1 + \b \w_1)^{1+q}\right]  \prod_{i \in [1,N] \backslash \overline{\tau}} 
	\E\left[(1 + \b \w_1)^{q}\right] \1{N \in \overline{\tau}} \right] \\
	&\leq \overline{\bE}\left[\prod_{i=1}^N 
	\E\left[(1 + \b \w_i)^{1+q}\right] \1{ i \in \overline{\tau}} \1{N \in \overline{\tau}}  \right].
	\end{split}
	\end{equation}
	In the last inequality, we neglect that contribution of sites for which the renewal process does not visit, since $\E\left[(1 + \b \w_i)^{q}\right]  \leq 1$. By Dominated Convergence we have
	\begin{equation}
	\lim_{\b \to 0} \E\left[(1 + \b \w_{(1,0)})^{1+q}\right]  = 1,
	\end{equation}
	since $q < \g -1$. For a given $\d > 0$, let $\b_0 = \b_0(\d)$ be such that, $\E\left[(1 + \b \w_{(1,0)})^{1+q} \right] \leq 1+\d$ for all $\b \leq \b_0$. Then, \eqref{key11}, \eqref{key22} and standard asymptotic results on the returning time of the simple random walk yields,
	\begin{equation} \label{ineq}
	\begin{split}
	\tilde{\E} \left[ \left( \bar{Z}_N^{\b, \tilde{\w}}   \right)^q \right] &\leq \left(1+\d\right)^N \overline{\bP}\left[
	N \in \overline{\tau} \right]^q \\
	&\leq \left(1+\d\right)^N \left( \dfrac{C_d}{N^{d/2}}\right)^q.
	\end{split}
	\end{equation}
	With the help of the following criterion, whose proof can be found in \cite[Proposition 2.5]{hubertbound}, we  see that this last bound suffices for our purpose. Let $A_n : = \tilde{\E}\left[\left( \bar{Z}_{n}^{\b, \tilde{\w}}\right)^q  \right]$.
	
	%
	\begin{lemma}
		If $k \in \N$ is such that
		\begin{equation} \label{rho}
		\rho := \tilde{\E}\left[ (1+\b \tilde{w}_{(1,0)})^q\right] 
		\sum_{n=k}^{\infty} \sum_{j=1}^{k-1} K(n-j)^q A_j < 1
		\end{equation}
		Then there exists $C = C(\rho, q, k, K(\cdot)) > 0$ such that
		\begin{equation} \label{final}
		A_N \leq C (K(N))^q,
		\end{equation}
	\end{lemma}
	for every $N \in \N$.
	%

	%
	%
	Assuming we can get \eqref{rho}, the proof is complete since we can use \eqref{final}, $K(N) \stackrel{N \to \infty}{\sim} \frac{C_d'}{N^{d/2}}$ since $d \geq 3$ and $q > 2/d$ to obtain \eqref{bound on pinning}. By \eqref{ineq} we see that
	\begin{equation}
	\begin{split}
	\rho &\leq (1+\d)^{k+1} \sum_{n=k}^{\infty} \sum_{j=1}^{k-1} K(n-j)^q \dfrac{(C_d)^q}{j^{dq/2}} \\
	&\leq C' (1+\d)^k \sum_{j=1}^{k-1} \dfrac{1}{j^{dq/2}}\frac{1}{(k-j)^{dq/2 - 1}} \\
	&\leq C' (1+\d)^k \left(
	\sum_{j=1}^{k/2} \dfrac{1}{j^{dq/2}}\frac{1}{(k-j)^{dq/2 - 1}}
	+ \sum_{j=k/2}^{k-1} \dfrac{1}{j^{dq/2}}\frac{1}{(k-j)^{dq/2 - 1}}
	\right) \\
	&\leq C' (1+\d)^k \left(
	\left( \sum_{j=1}^{\infty} \dfrac{1}{j^{dq/2}}\right) \frac{1}{(k/2)^{dq/2 - 1}}
	+  \dfrac{1}{(k/2)^{dq/2 - 1}}
	\right),
	\end{split}
	\end{equation}
	for some constant $C'$. Then, for a given $\eps > 0$ we can choose $k \in \N$ such that 
	\begin{equation}
	\left( \sum_{j=1}^{\infty} \dfrac{1}{j^{dq/2}}\right) \frac{1}{(k/2)^{dq/2 - 1}}
	+  \dfrac{1}{(k/2)^{dq/2 - 1}} < \eps/(2C') 
	\end{equation}
	and then choose $\d > 0$ such that $(1+\d)^k < 2$. This proves that we can made the value of  $\rho$ arbitrary small.
\end{proof}
\section{Lower bound.}
In this section, we show a lower bound for the free energy, assuming $\g < \g_c$. This completes the proof of Theorem \ref{1st}.
\begin{proposition}
		Given $\eps > 0$, under usual hypothesis on the environment's distribution, $\g < \g_c(d)$ and $d \geq 3$ we have
	\begin{equation}\label{1stlowebound}
	p(\b) \geq - C_\eps \b^{\a-\eps},
	\end{equation}
for all $\b > 0$ sufficiently small.
\end{proposition}

\begin{proof}
Consider the following partition function of a truncated version of the environment:  
\begin{equation} \label{bounded}
\breve{Z}_N^{\b, \w} := \bE \left[\prod_{i=1}^N \frac{1+\b \w_{i,S_i}\wedge\b^{-\kappa}}{c_\b } \right], 
\end{equation}
where $c_\b := \E\left(1+\b \w_{1,0}\wedge\b^{-\kappa} \right) $ and $\kappa>0$ is a constant to be fixed soon. Observe that,
\begin{equation}
\lim_{N \to \infty} \frac{1}{N} \log Z_N^{\b, \w} \geq \lim_{N \to \infty} \frac{1}{N} \log Z_N^{\b, \w \wedge \b^{-\kappa}} =  \log c_\b + \lim_{N \to \infty} \frac{1}{N} \log \breve{Z}_N^{\b, \w}.
\end{equation}
This implies that, by showing that
\begin{equation} \label{free3}
\lim_{N \to \infty} \frac{1}{N} \log \breve{Z}_N^{\b, \w} = 0,
\end{equation}
we obtain the result since we can choose $\kappa \in \left(\frac{\g_c}{\g_c - \g} - \frac{\eps}{\g - 1}, \frac{\g_c}{\g_c - \g} \right) $ to get 
\begin{equation}
 \log c_\b \geq - C_\kappa\b^{\kappa(\g-1) + 1} \geq -C_\kappa \b^{\frac{\g(\g_c-1)}{\g_c - \g} - \eps},
\end{equation}
for a constant $C_\kappa>0$ that depends also on the environment's distribution.

To prove \eqref{free3}, we will adapt the same strategy as Section $4$: showing that
	\begin{equation} \label{ui2}
	\sup_{N\geq 1} \E \left[ \left( \breve{Z}_N^{\b, \w} \right)^{1+q} \right] < \infty,
	\end{equation}    
	for some $q \in (\g_c - 1 ,1)$. As we did before, we can write this expectation as
	\begin{equation}
	\begin{split}
	\E \left[ \left( \breve{Z}_N^{\b, \w} \right)^{1+q} \right] &=  \E \left[\bE \left[\prod_{i=1}^N  \frac{1+\b \w_{i,S_i}\wedge\b^{-y}}{c_\b  } \right]  \left( \breve{Z}_N^{\b, \w} \right)^{q} \right] \\
	&= \bE \left[\E \left[\left( \breve{Z}_N^{\b, \w} \right)^{q} \prod_{i=1}^N  \frac{1+\b \w_{i,S_i}\wedge\b^{-y}}{c_\b }  \right]   \right].
	\end{split} 
	\end{equation}
Then, considering the IID random variables $\tilde{\w} = \{\tilde{\w}_{n,z}: n \in \N ,z \in \Z^{d}\}$ from a probability space $ (\tilde{\Lambda}, \tilde{\cF}, \tilde{\P})$ of distribution given by:
	\begin{equation}
	\tilde{\P}\left( \tilde{\w}_{1,0} \in A \right) = \E\left[  \frac{1+\b \w_{1,0}\wedge\b^{-\kappa}}{c_\b } \1{\w_{1,0} \in A} \right],  
	\end{equation}
	we could express the expectation above as
	\begin{equation}
	\E \left[\left( \breve{Z}_N^{\b, \w} \right)^{q} \prod_{i=1}^N  \frac{1+\b \w_{i,S_i}\wedge\b^{-y}}{c_\b}  \right]  = \E \otimes \tilde{\E} \left[\left( Z_N^{\b, \hat{\w}^S} \right)^{q}  \right],  
	\end{equation}
	where for all $i,z$ we use the notation introduced in Equation \eqref{quenched}: 
	\begin{equation}
	\hat{\w}_{i,z}^S := \w_{i,z}\1{z \not= S_i} + \tilde{\w}_{i,z}\1{z = S_i}.
	\end{equation}
	Let us denote $\w'_{i,S'_i} := \w_{i,S'_i}\wedge \b^{-y} $ and  $\tilde{\w}'_{i,S'_i} := \tilde{\w}_{i,S'_i}\wedge \b^{-y}$.  Fubini's Theorem and Jensen's inequality yields
	
	\begin{equation}
	\begin{split}
	\E \left[ \left( \breve{Z}_N^{\b, \w} \right)^{1+q} \right] &\leq  \bE \otimes \tilde{\E}\left[ \left( \E \left[ \breve{Z}_N^{\b, \hat{\w}^S} \right]\right) ^q    \right]\\
	&= \bE \otimes \tilde{\E}\left[ \left( \E \otimes \bE' \prod_{i=1}^{N}\frac{1+\b \w'_{i,S'_i} \1{S'_i \not= S_i} + \b \tilde{\w}'_{i,S'_i} \1{S'_i = S_i} }{c_\b} \right)^q    \right]  \\
	&\leq \bE \otimes \tilde{\E}\left[  \left( \bE' \prod_{\substack{1\leq i\leq N \\ S_i = S'_i}}\frac{1+\b \tilde{\w}'_{i,S_i}}{c_\b} \right)^q \right],
	\end{split}
	\end{equation}
	where $  \left( \O',\bP', S'\right) $ is an independent copy of $  \left( \O,\bP, S\right) $. As before, we replace $\tilde{\w}'_{i,S_i}$ by $\tilde{\w}'_{i,0}$ (which we now simply denote by $\tilde{\w}'_{i}$) and apply Jensen's Inequality one more time to get,
	\begin{equation}
	\E \left[ \left( \breve{Z}_N^{\b, \w} \right)^{1+q} \right] \leq
	\tilde{\E}\left[  \left(\bE'\otimes \bE  \prod_{\substack{1\leq i\leq N \\ S_i = S'_i}}\frac{1+\b \tilde{\w}'_{i}}{c_\b}    \right)^q    \right].
	\end{equation}
	Using the same notation defined in \eqref{renewal} and by the argument in \eqref{pinning}, we are left with showing that
	\begin{equation}
	\sum_{N=1}^\infty \tilde{\E} \left[\left( \bar{\bE} \prod_{\substack{1\leq i\leq N \\ i \in \bar{\t}}}\frac{1+\b \tilde{\w}'_i }{c_\b} \1{N \in \bar{\t}} \right) ^q  \right] < \infty.
	\end{equation}
	Let us call
	\begin{equation}
	 \breve{Z}_{N,b}^{\b, \tilde{\w}'} := \bar{\bE} \left[  \prod_{\substack{1\leq i\leq N \\ i \in \bar{\t}}}\frac{1+\b \tilde{\w}'_i }{c_\b} \1{N \in \bar{\t}}\right] .
	\end{equation}
	Define the function
	\begin{equation}
	\begin{split}
	h:\R &\to \R \\
	x &\mapsto \left( \frac{1+\b x}{c_\b}\right) ^{-q(1-q)}.
	\end{split}
	\end{equation}
	By Holder's Inequality, 
	\begin{equation}
	\tilde{\E} \left[ \left(  \breve{Z}_{N,b}^{\b, \tilde{\w}'}   \right)^q \right] \leq 
	\tilde{\E} \left[ \left(  \prod_{i=1}^{N}h(\tilde{\w}'_{i})^{-1} \right)^{\frac{1}{1-q}} \right]^{1-q} 
	\tilde{\E} \left[ \left(  \prod_{i=1}^{N}h( \tilde{\w}'_{i}) \right)^{1/q}  \breve{Z}_{N,b}^{\b, \tilde{\w}'} \right]^{q}. 
	\end{equation}
	For the first expectation we have,
	\begin{equation}\label{key112}
	\begin{split}
	\tilde{\E} \left[ \left(  \prod_{i=1}^{N}h(\tilde{\w}'_i)^{-1} \right)^{\frac{1}{1-q}} \right]^{1-q} &= \tilde{\E} \left[ \left( \frac{1 + \b \tilde{\w}_1 \wedge \b^{-\kappa}}{c_\b}\right)^q \right]^{N(1-q)} \\
	&= \E\left[\left( \frac{1 + \b \w_1\wedge \b^{-\kappa}}{c_\b}\right)^{1+q} 
	\right]^{N(1-q)}.
	\end{split}
	\end{equation}
	For the second expectation,
	\begin{equation} \label{key222}
	\begin{split}
	\tilde{\E} \left[ \left(  \prod_{i=1}^{N}h( \tilde{\w}'_i) \right)^{1/q} \breve{Z}_{N,b}^{\b, \tilde{\w}'} \right]^{q} &= 
	\tilde{\E} \left[   \prod_{i=1}^{N} \left( \frac{1 + \b \tilde{\w}_i\wedge \b^{-\kappa}}{c_\b} \right)^{-(1-q)} 
\bar{\bE} \left[ \prod_{\substack{1\leq i\leq N \\ i \in \bar{\t}}}\frac{1+\b \tilde{\w}_i \wedge \b^{-\kappa}}{c_\b} \1{N \in \bar{\t}} \right]
	\right]^{q} \\ 
	&\leq \bar{\bE} \left[ \prod_{\substack{1\leq i\leq N \\ i \in \bar{\t}}}\E\left[\left( \frac{1 + \b \w_1\wedge \b^{-\kappa}}{c_\b}\right)^{1+q} 
	\right] \1{N \in \bar{\t}} \right]^q.
	\end{split}
	\end{equation}
	As we did in the previous section, we now need to show that
	\begin{equation}
	\E\left[\left( \frac{1 + \b \w_1\wedge \b^{-\kappa}}{c_\b}\right)^{1+q} 
	\right]
	\end{equation}
	is arbitrarily close to one, for all $\b$ sufficiently small. This is done in the next Lemma. After this, the rest of the proof follows the exact same lines as the proof from Section 4.
	
	\begin{lemma}
		For some $q > \g_c - 1$, we have
		\begin{equation}
			\lim_{\b \to 0} \E\left[\left( \frac{1 + \b \w_1\wedge \b^{-\kappa}}{c_\b}\right)^{1+q} 
		\right] = 1. 
		\end{equation}
	\end{lemma}  
\begin{proof}
Let us called $\w = \w_1$. Notice that we only need to focus on the numerator of the fraction since $c_\b =  \E \left[ 1 + \b \w \wedge \b^{-\kappa} \right]  \to 1$, as $\b \to 0$. For the numerator we have that, for a fixed $\d > 0$,
\begin{equation}
\begin{split}
\E\left[\left(1 + \b \w \wedge \b^{-y}\right)^{1+q} 
\right] &=\E\left[\left(1 + \b \w \right)^{1+q} 
\1{\w \leq \d \b^{-1}} \right] + \E\left[\left(1 + \b \w \right)^{1+q} 
\1{\d \b^{-1} \leq \w \leq \b^{-y}}\right] \\
&+ \E\left[\left(1 + \b^{-y+1}\right)^{1+q} 
 \1{\b^{-y} \leq \w } \right].
\end{split}
\end{equation}
For the first summand, we have that
\begin{equation}
\E\left[\left(1 + \b \w \right)^{1+q} 
\1{\w \leq \d \b^{-1}} \right] \leq (1+\d)^{1+q}.
\end{equation}
For the third summand, we have that
\begin{equation}
\E\left[\left(1 + \b^{-\kappa+1}\right)^{1+q} 
\1{\b^{-\kappa} \leq \w } \right] \leq C \b^{(-\kappa+1)(1+q) + \kappa \g}.
\end{equation}
Since $\kappa < \frac{\g_c}{\g_c - \g}$ is fixed,  we can choose the value of $q$ sufficiently close to $\g_c - 1$ so that the exponent $(-\kappa+1)(1+q) + \kappa \g > 0$, making the third summand arbitrarily close to zero. Finally, using the identity 
\begin{equation}
\E \left[ X^{p} \1{a\leq X \leq b} \right] = a^p\P\left[X \geq a \right] - 
b^p\P\left[X \geq b \right] + p \int_{a}^{b} z^{p-1}\P\left[X \geq z \right]dz,
\end{equation}
valid for any random variable $X \geq 0$, we obtain for the second summand,	
\begin{equation}
\begin{split}
\E\left[\left(1 + \b \w \right)^{1+q} \1{\d \b^{-1} \leq \w \leq \b^{-\kappa}}\right] &\leq (\d + 1)^{1+q}\P \left[\b \w \geq  \d\right] + \int_{\d}^{\b^{-\kappa+1}+1} (1+q)z^q\P\left[1+\b\w \geq z \right] dz \\
&\leq C (\d + 1)^{1+q} \d^{-\g} \b^\g +  C \b^\g \int_{\d}^{\b^{-\kappa+1}} (z^{q-\g} + z^{-\g})dz \\
&\leq C (\d + 1)^{1+q} \d^{-\g} \b^\g +  C \b^{\g + (-\kappa +1)(q - \g +1)} + C\b^\g \d^{-\g + 1}. 
\end{split}
\end{equation}
Using one more time the fact that $(-\kappa+1)(1+q) + \kappa \g > 0$ we can make the last term arbitrarily small, concluding the proof of the Lemma.	
\end{proof}	
\end{proof}	
\appendix
\section{properties of the free energy.}
\begin{theorem} \label{freenergy}
%
The function $p : [0,1) \to \R$, defined as
\begin{equation} \label{cfe}
 p(\beta) = \lim \dfrac{1}{N} \ln  Z_{N}^{\beta,\omega},
	\end{equation}
is continuous and non-increasing. 
\end{theorem}
\begin{proof}

Notice that the function $\b \mapsto p_N(\b):= \dfrac{1}{N} \E\left[ \log Z_N^{\b,\w} \right] $ is differentiable and 
\begin{equation}
\begin{split}
\dfrac{\partial}{\partial \b} p_N(\b) &= 
\dfrac{1}{N} \E\left[\dfrac{1}{Z_N^{\b,\w}} \bE \left[ \sum_{i=1}^{N} \w_{(i,S_i)} \prod_{j \in\llb1,N \rrb \backslash \{i\}} (1+\b \w_{(j,S_j)}) \right] \right] \\  
&= \dfrac{1}{N} \E\left[  \bE_N^{\b,\w} \left[ \sum_{i=1}^{N} \dfrac{\w_{(i,S_i)}}{1+\b \w_{(i,S_i)}}\right]\right].
\end{split}
\end{equation}
Then $\left| \dfrac{\partial}{\partial \b} p_N(\b) \right| \leq K_\b := \max\left\lbrace  \dfrac{1}{\b},\dfrac{1}{1-\b}\right\rbrace $ which implies  that the limit $p(\b)$ is a continuous function in $(0,1)$. Also, since the functions
\begin{equation}
\w \mapsto \dfrac{1}{Z_N^{\b,\w}},
\end{equation}
and
\begin{equation}
\w \mapsto \bE \left[ \sum_{i=1}^{N} \w_{(i,S_i)} \prod_{j \in\llb1,N \rrb \backslash \{i\}} (1+\b \w_{(j,S_j)})\right] 
\end{equation}
are decreasing and increasing respectively, applying FKG inequality, we have that
\begin{equation}
\dfrac{\partial}{\partial \b} p_N(\b) \leq 
\dfrac{1}{N} \E\left[\dfrac{1}{Z_N^{\b,\w}} \right] \E\left[ \bE\left[  \sum_{i=1}^{N} \w_{(i,S_i)} \prod_{j \in\llb1,N \rrb \backslash \{i\}} (1+\b \w_{(j,S_j)}) \right]\right]  = 0,
\end{equation}
which implies that $p(\b)$ is non-increasing.
\end{proof}

\end{document}